\documentclass[reqno]{amsart}
\usepackage{amsmath,amssymb,amsthm}
\usepackage[hidelinks]{hyperref}
\hypersetup{pdfauthor={Hy P. G. Lam},pdftitle={Zero-energy scattering and the real Bers image on the line}}

\newtheorem{theorem}{Theorem}[section]
\newtheorem{proposition}[theorem]{Proposition}
\newtheorem{lemma}[theorem]{Lemma}
\newtheorem{corollary}[theorem]{Corollary}
\theoremstyle{definition}
\newtheorem{definition}[theorem]{Definition}
\newtheorem{remark}[theorem]{Remark}
\newtheorem{thmintro}{Theorem}
\newtheorem{corintro}[thmintro]{Corollary}

\newcommand{\R}{\mathbb{R}}
\newcommand{\C}{\mathbb{C}}
\newcommand{\Sw}{\mathcal{S}}
\newcommand{\SwR}{\mathcal{S}(\R)}
\newcommand{\SwRreal}{\mathcal{S}(\R;\R)}
\newcommand{\SwRzero}{\mathcal{S}_0(\R;\R)}
\newcommand{\Dg}{\operatorname{Diff}_{\mathcal S}(\R)}
\newcommand{\dif}{\,\mathrm{d}}
\newcommand{\im}{\operatorname{Im}}
\newcommand{\Wr}{W}
\newcommand{\sla}{\mathcal R_{\!B}}
\newcommand{\Bers}{\beta}
\newcommand{\RRset}{\mathfrak R_0}

\begin{document}

\title[Zero-energy scattering and the real Bers image]{Zero-energy scattering and the real Bers image on the line}

\author{Hy P. G. Lam}
\address{Department of Mathematical Sciences, Worcester Polytechnic Institute, Worcester, MA 01609, USA}
\email{hylam.math@gmail.com, hlam@wpi.edu}


\begin{abstract}
Let $\Dg$ be the group of orientation-preserving diffeomorphisms $\varphi$ of the line with $\varphi'-1$ Schwartz and $\varphi(x)-x\to0$ as $x\to-\infty$, and let $\Bers(\varphi)=\tfrac12S(\varphi)$ be half the Schwarzian derivative. We determine the image of $\Bers$. The coordinate $w_\varphi=\tfrac12(\log\varphi')'$ identifies $\Dg$ with the zero-mean hyperplane $\SwRzero$ and turns $\Bers$ into $w\mapsto w'-w^2$, so the question is which real Schwartz $q$ equal $w'-w^2$ for some $w$ of zero mean. Nonnegativity of the Schr\"odinger operator $H_q=-\partial_x^2-q$ decides which $q$ admit such a $w$ at all, and says nothing about the mean of $w$. The mean is a scattering invariant. Let $T_q$ and $R_q$ be the transmission and reflection coefficients of $H_q$. For every real $w\in\SwR$, with $q=w'-w^2$, we prove
\[
T_q(0)=\operatorname{sech}\!\left(\int_\R w\dif x\right),
\qquad
R_q(0)=-\tanh\!\left(\int_\R w\dif x\right),
\]
so $\int_\R w\dif x$ is read off the scattering matrix at zero energy. Thus $q\in\Bers(\Dg)$ if and only if $H_q$ has no negative eigenvalues and $R_q(0)=0$, equivalently $T_q(0)=1$, and $\varphi\mapsto R_{\Bers(\varphi)}$ is a bijection onto an explicit set of Schwartz reflection coefficients. We also compute the differential of $\Bers$, whose range depends on the ambient topology. On the Schwartz space, that range is closed and split of codimension 2, with normal functionals the first variations of the Wronskian of the two zero-energy Jost solutions and of $R_q(0)$. In every $W^{k,1}$ realization, the second functional is unbounded. The range is then a dense proper subspace of the kernel of the first, hence neither closed nor split, and the operator has no bounded inverse on its range.
\end{abstract}

\maketitle

\section{Introduction}

Let
\[
\Dg=\left\{\varphi\in\operatorname{Diff}^{+}(\R)\ \middle|\ \varphi'-1\in\SwRreal,
\quad \varphi(x)-x\to0\text{ as }x\to-\infty\right\}
\]
and define the real Bers map by
\[
\Bers(\varphi)=\frac12S(\varphi),
\qquad
S(\varphi)=\frac{\varphi'''}{\varphi'}-\frac32\left(\frac{\varphi''}{\varphi'}\right)^2.
\]
Lu introduced this map on a larger decay-controlled group and asked for an exact description of its image \cite[Remark 3.4.4]{Lu25}. We solve the Schwartz-class form of that problem and compute the range of the differential.

Write $u=\log\varphi'$ and $w=u'/2$. Then
\begin{equation}\label{eq:potentialform}
\Bers(\varphi)=w'-w^2.
\end{equation}
More generally, if $q,w\in\SwRreal$ satisfy
\[
q=w'-w^2,
\]
we call $w$ a \emph{Schwartz Riccati representative} of $q$. The map \[
\Theta\colon\Dg\longrightarrow\SwRzero,
\qquad
\Theta(\varphi)=w_\varphi,
\] identifies $\Dg$ with the closed hyperplane
\[
\SwRzero=\left\{w\in\SwRreal\ \middle|\ \int_\R w\dif x=0\right\}.
\]
Thus the image problem is to determine which real Schwartz potentials admit a Schwartz Riccati representative of zero mean.

For a real Schwartz potential $q$, let $f_\pm(\cdot,k)$ be the Jost solutions of $-y''-qy=k^{2}y$ fixed by $f_\pm(x,k)=e^{\pm ikx}(1+o(1))$ as $x\to\pm\infty$, write $\Wr(f,g)=fg'-f'g$, let $a_q$ and $b_q$ be the scattering coefficients in the Jost relation of Section \ref{sec:jost}, and set
\[
T_q(k)=\frac1{a_q(k)},
\qquad
R_q(k)=\frac{b_q(k)}{a_q(k)}.
\]
Here $R_q$ is the right reflection coefficient of $H_q=-\partial_x^2-q$. Our first result recovers $\int_\R w\dif x$ from the scattering data at zero energy.

\begin{thmintro}[Theorem \ref{thm:threshold}]\label{thmA}
Let $w\in\SwRreal$, set $q=w'-w^2$, and write $I(w)=\int_\R w\dif x$. Then $H_q\geq0$, the two zero-energy Jost solutions are linearly dependent, and
\[
a_q(0)=\cosh I(w),
\qquad
b_q(0)=-\sinh I(w),
\]
\[
T_q(0)=\operatorname{sech} I(w),
\qquad
R_q(0)=-\tanh I(w).
\]
\end{thmintro}

The factorization
\[
H_q=(-\partial_x+w)(\partial_x+w)
\]
shows only that $H_q$ is nonnegative. Theorem \ref{thmA} identifies the scalar information not seen by positivity with a threshold scattering invariant, and yields the following answer to the image problem.

\begin{thmintro}[Theorem \ref{thm:image}]\label{thmB}
Let $q\in\SwRreal$. The following are equivalent.
\begin{enumerate}
\item $q=\Bers(\varphi)$ for a unique $\varphi\in\Dg$.
\item $q=w'-w^2$ for a unique $w\in\SwRzero$.
\item $H_q$ has no negative eigenvalues and $R_q(0)=0$.
\item $H_q$ has no negative eigenvalues and $T_q(0)=1$.
\end{enumerate}
When these conditions hold, the zero-energy right Jost solution satisfies $f_+(\cdot,0)>0$ and $f_+(x,0)\to1$ as $x\to\pm\infty$, and
\[
\varphi'(x)=f_+(x,0)^{-2}.
\]
\end{thmintro}

The full Miura range is characterized by nonnegativity of the Schr\"odinger form \cite{KPST05}. Frayer, Hryniv, Mykytyuk, and Perry proved the exceptional inverse-scattering bijection for unique integrable Riccati representatives \cite[Theorem 1.1 and Corollary 1.2]{FHMP09}; weighted Sobolev mapping properties were developed in \cite{HMP10}, building on Zhou's $L^2$-Sobolev bijectivity theorem \cite{Zhou98}. Theorem \ref{thmA} cuts out the normalized real Bers image inside this known Miura class by the equation $R_q(0)=0$.

Define
\[
\RRset=\left\{R\in\SwR\ \middle|\
\overline{R(-k)}=R(k),\quad \lVert R\rVert_\infty<1,\quad R(0)=0\right\}.
\]

\begin{corintro}[Corollary \ref{cor:reflection}]\label{thmC}
The map
\[
\varphi\longmapsto R_{\Bers(\varphi)}
\]
is a bijection from $\Dg$ onto $\RRset$.
\end{corintro}

The local image has two independent normal directions. Fix $q=\Bers(\varphi)$, put $u=\log\varphi'$ and $\psi=e^{-u/2}$, let
\[
\widetilde\psi(x)=\psi(x)\int_{x_0}^{x}\psi(t)^{-2}\dif t,
\qquad \Wr(\psi,\widetilde\psi)=1,
\]
be the second zero-energy solution, and define
\[
\Lambda_1(v)=\int_\R\psi^{2}v\dif x,
\qquad
\Lambda_2(v)=\int_\R\psi\widetilde\psi\,v\dif x .
\]

\begin{thmintro}[Theorem \ref{thm:differential}]\label{thmD}
The range of $D\Bers_u\colon\SwRreal\to\SwRreal$ is $\ker\Lambda_1\cap\ker\Lambda_2$, and that range is closed and split of codimension two. The functional $\Lambda_1$ is the first variation of $\Wr\bigl(f_-(\cdot,0),f_+(\cdot,0)\bigr)$, and on $\ker\Lambda_1$ the functional $\Lambda_2$ is the first variation of $R_q(0)$. For every $k\geq0$, the range of
\[
D\Bers_u\colon W^{k+2,1}(\R)\longrightarrow W^{k,1}(\R)
\]
is not closed, and the closure of that range is $\ker\Lambda_1$.
\end{thmintro}

The functional $\Lambda_2$ integrates against $\psi\widetilde\psi$, and $\psi\widetilde\psi$ grows linearly, so $\Lambda_2$ is continuous on $\SwR$ and unbounded on every $W^{k,1}(\R)$, none of which controls a spatial moment. The two constraints $\Lambda_1(v)=0$ and $\Lambda_2(v)=0$ are the same in both settings. What differs is the subspace those constraints define, because a continuous linear functional has closed kernel while a discontinuous one has dense kernel. In $\SwR$ the two constraints cut out a closed subspace of codimension two, and in $W^{k,1}(\R)$ the range is instead a dense proper subspace of the closed hyperplane $\ker\Lambda_1$. The normalization that Theorem \ref{thmB} pins down globally by $R_q(0)=0$ is therefore visible infinitesimally only in a topology that sees spatial moments, which is one reason the Schwartz class rather than the larger class of \cite{Lu25} is the right setting here.

\section{The diffeomorphism group and the real Bers map}\label{sec:group}

Throughout, $\SwR$ is the Schwartz space, $\SwRreal$ its real subspace, and $\widehat f(\xi)=\int_\R f(x)e^{-i\xi x}\dif x$. All diffeomorphisms are orientation-preserving.

The group $\Dg$ is the Schwartz analogue of the group $\operatorname{Diff}_{-\infty}(\R)$ of Bauer, Lu and Maor \cite{BLM22}, where $\varphi'-1$ is required to lie in $W^{\infty,1}(\R)$, and $\operatorname{Diff}_{-\infty}(\R)$ is the setting of \cite{Lu25}.

Since $\varphi'(x)\to1$ as $x\to\pm\infty$ and $\varphi'$ is positive and continuous, $\inf\varphi'>0$, so $u=\log\varphi'$ is defined, and writing $u=g\int_0^1(1+sg)^{-1}\dif s$ with $g=\varphi'-1\in\Sw$ shows $u\in\SwRreal$. Conversely, $e^{u}-1\in\Sw$ for $u\in\Sw$.

\begin{lemma}\label{lem:chart}
The map $\Phi_\infty(\varphi)=\log\varphi'$ is a bijection from $\Dg$ onto $\SwRreal$ with inverse $u\mapsto\varphi$, $\varphi(x)=x+\int_{-\infty}^x(e^{u(t)}-1)\dif t$. One has $\Phi_\infty(\varphi\circ\psi)=\Phi_\infty(\varphi)\circ\psi+\Phi_\infty(\psi)$, and $\Dg$ is a group.
\end{lemma}

The corresponding statement in the $W^{\infty,1}$-setting is \cite[Corollary 3.1.2]{Lu25}. The $L^p$ charts $\varphi\mapsto p\bigl((\varphi')^{1/p}-1\bigr)$, of which $\Phi_\infty$ is the limit, are \cite[Theorem 2.2.1]{BLM22}.

\begin{proof}
If $\psi\in\Dg$, then
\[
\psi(x)=x+c_\pm+O(|x|^{-N})
\quad\text{as }x\to\pm\infty
\]
for every $N$, with $c_-=0$, and all derivatives of order at least two are Schwartz. Hence composition with $\psi$ preserves $\Sw$. The identities
\[
(\varphi\circ\psi)'-1=\big((\varphi'-1)\circ\psi\big)\psi'+(\psi'-1)
\]
and
\[
(\varphi^{-1})'-1=\left(\frac{1-\varphi'}{\varphi'}\right)\circ\varphi^{-1}
\]
show that composition and inversion preserve the defining decay. The normalization at $-\infty$ is immediate, and the chart formulas follow from the chain rule.
\end{proof}

\begin{lemma}\label{lem:riccati-chart}
The map
\[
\Theta\colon\Dg\longrightarrow\SwRzero,
\qquad
\Theta(\varphi)=\frac12(\log\varphi')'
\]
is a bijection. Its inverse is
\[
u_w(x)=2\int_{-\infty}^{x}w(t)\dif t,
\qquad
\varphi_w(x)=x+\int_{-\infty}^{x}\bigl(e^{u_w(t)}-1\bigr)\dif t.
\]
For $w=\Theta(\varphi)$ one has $\Bers(\varphi)=w'-w^2$.
\end{lemma}

\begin{proof}
If $\varphi\in\Dg$ and $u=\log\varphi'$, then $u\in\SwRreal$ by Lemma \ref{lem:chart}, so $w=u'/2$ is Schwartz and $\int_\R w\dif x=\tfrac12\bigl(u(+\infty)-u(-\infty)\bigr)=0$. Conversely, if $w\in\SwRzero$, then $u_w'=2w$. The zero-mean condition gives $u_w(x)=-2\int_x^\infty w(t)\dif t$ for $x>0$, so $u_w$ and all its derivatives decay rapidly as $x\to\pm\infty$. Hence $u_w\in\SwRreal$, and Lemma \ref{lem:chart} produces $\varphi_w\in\Dg$. The two constructions are inverse to one another, and the last identity follows by differentiation.
\end{proof}

\begin{definition}[{Lu \cite[Section 3.4]{Lu25}}]\label{def:bers}
The real Bers map is
\[
\Bers\colon\Dg\to\SwRreal,\qquad \Bers(\varphi)=\tfrac12 S(\varphi)=\tfrac12 u''-\tfrac14(u')^{2},\qquad u=\log\varphi' .
\]
\end{definition}

The normalization is chosen so that \eqref{eq:potentialform} holds with $w=u'/2$, hence $\Bers(\varphi)=w'-w^{2}$ is, up to sign, the Miura transform of $w$ \cite{Miura68}. Lu proved the corresponding statement in the $W^{\infty,1}$ setting \cite[Theorem 3.4.3]{Lu25}. We record its Schwartz specialization. The underlying projective correspondence is classical; see \cite{OT05}.

\begin{proposition}[{Schwartz specialization of \cite[Theorem 3.4.3]{Lu25}}]\label{prop:injective}
$\Bers$ is injective on $\Dg$ and
\[
\Bers(\Dg)
=\big\{\tfrac12 f''-\tfrac14(f')^2\;\big|\;f\in\SwRreal\big\}
=\big\{w'-w^2\;\big|\;w\in\SwRzero\big\}.
\]
Moreover, $\psi=e^{-u/2}$ satisfies $H_{\Bers(\varphi)}\psi=0$, $\psi>0$, $\psi(x)\to1$ as $x\to\pm\infty$, and
\begin{equation}\label{eq:reconstruction}
\varphi(x)=x+\int_{-\infty}^{x}\big(\psi(t)^{-2}-1\big)\dif t .
\end{equation}
\end{proposition}

\begin{proof}
The image formulas follow from Lemmas \ref{lem:chart} and \ref{lem:riccati-chart}. For injectivity, suppose $S(\varphi)=S(\psi_1)$. The cocycle identity $S(\varphi\circ\psi)=(S(\varphi)\circ\psi)(\psi')^{2}+S(\psi)$ applied to $\varphi=(\varphi\circ\psi_1^{-1})\circ\psi_1$ gives $S(\varphi\circ\psi_1^{-1})=0$, so $g=\varphi\circ\psi_1^{-1}$ is a M\"obius transformation; a M\"obius map $g(x)=(\alpha x+\beta)/(\gamma x+\delta)$ which is a diffeomorphism of $\R$ has no pole, so $\gamma=0$ and $g(x)=ax+b$ with $a>0$. Then $\varphi'=a\,\psi_1'$ and letting $x\to+\infty$ gives $a=1$. The normalization at $-\infty$ gives $b=0$, hence $\varphi=\psi_1$. The factorization $H_q=A^{*}A$ with $A=\partial+w$ yields $A\psi=(-\tfrac12u'+\tfrac12u')\psi=0$, hence $H_q\psi=0$. Positivity and the limits are clear, and \eqref{eq:reconstruction} inverts $\psi=(\varphi')^{-1/2}$.
\end{proof}

\begin{remark}\label{rem:miura}
The Miura characterization detects nonnegativity but not the two-sided normalization of the positive zero-energy solution. Every $q$ in the image satisfies $\int_\R q\dif x=-\int_\R w^{2}\dif x\leq0$, which is the obstruction used in \cite[Theorem 3.4.3]{Lu25} to show that the image is not open. Theorem \ref{thm:image} shows that the exact normalization is $R_q(0)=0$.
\end{remark}

\section{Jost solutions and the scattering coefficients}\label{sec:jost}

The facts in this section are classical. We use the normalizations of Faddeev \cite{Fad64,Fad74}, Marchenko \cite{Mar86}, and Deift and Trubowitz \cite{DT79}. We record only the statements needed below. Let $q\in\SwRreal$ and
\[
H_q=-\partial_x^{2}-q \qquad\text{on } L^{2}(\R),
\]
a short-range operator with essential spectrum $[0,\infty)$ and finitely many simple negative eigenvalues, the bound states. For $\im k\geq0$, the Jost solutions $f_{\pm}(x,k)$ of
\begin{equation}\label{eq:jost}
-y''-q\,y=k^{2}y
\end{equation}
are fixed by $f_{+}(x,k)=e^{ikx}(1+o(1))$ as $x\to+\infty$ and $f_{-}(x,k)=e^{-ikx}(1+o(1))$ as $x\to-\infty$. For real $k\neq0$,
\begin{equation}\label{eq:scatterrel}
f_{-}(x,k)=a(k)\,f_{+}(x,-k)+b(k)\,f_{+}(x,k),
\end{equation}
and $T_q(k)=1/a(k)$, $R_q(k)=b(k)/a(k)$. Thus $R_q$ is the right reflection coefficient of $H_q$. From \eqref{eq:scatterrel} and $\Wr\big(f_+(\cdot,-k),f_+(\cdot,k)\big)=2ik$,
\begin{equation}\label{eq:ab}
a(k)=\frac{\Wr\big(f_{-}(\cdot,k),f_{+}(\cdot,k)\big)}{2ik},
\qquad
b(k)=\frac{\Wr\big(f_{+}(\cdot,-k),f_{-}(\cdot,k)\big)}{2ik}.
\end{equation}
One has $|a|^{2}=1+|b|^{2}$, so $|T_q|^{2}+|R_q|^{2}=1$, together with $a(-k)=\overline{a(k)}$ and $b(-k)=\overline{b(k)}$; the function $a$ extends holomorphically to $\im k>0$ with $a(k)\to1$ at infinity, and its zeros there are exactly the points $i\kappa_j$ with $-\kappa_j^{2}$ a bound state.

\begin{definition}\label{def:exceptional}
A real short-range potential $q$ is called \emph{exceptional at zero energy} if the Wronskian 
\[
\Wr\bigl(f_-(\cdot,0),f_+(\cdot,0)\bigr)=0.
\]
Equivalently, the two zero-energy Jost solutions are linearly dependent, so there is a unique real number $c(q)\neq0$ such that
\[
f_-(\cdot,0)=c(q)\,f_+(\cdot,0).
\]
If the Wronskian is nonzero, $q$ is called \emph{generic at zero energy}.
\end{definition}

\begin{lemma}\label{lem:regularity}
Let $q\in\SwRreal$. For each fixed $x$, the functions
\[
m_\pm(x,k)=e^{\mp ikx}f_\pm(x,k)
\]
are holomorphic for $\im k>0$ and smooth up to the real axis from above. Their $k$ derivatives are uniform when $x$ is restricted to the corresponding half-line. As $|k|\to\infty$ along the real axis, $b$ and all of its derivatives decay faster than every power, while
\[
\partial_k^j\bigl(a(k)-1\bigr)=O\bigl(|k|^{-j-1}\bigr)
\]
for every $j\geq0$.

At $k=0$, exactly one of the cases in Definition \ref{def:exceptional} occurs. If $q$ is generic, then
\[
T_q(0)=0,\qquad R_q(0)=-1.
\]
If $q$ is exceptional, then $a$ and $b$ are smooth at zero.
\end{lemma}

\begin{proof}
The right Jost factor satisfies
\[
m_+(x,k)=1-\int_x^\infty D_k(t-x)q(t)m_+(t,k)\dif t,
\qquad
D_k(s)=\frac{e^{2iks}-1}{2ik}.
\]
For $s\geq0$ and $\im k\geq0$, one has
\[
|D_k(s)|\leq \min\{s,|k|^{-1}\}.
\]
Each $k$ derivative of $D_k$ is bounded by a constant times a power of $s$. The Schwartz decay of $q$ therefore permits differentiation of the Volterra series to every order. The left Jost factor is treated in the same way.

For real $k$ one has $f_\pm(x,-k)=\overline{f_\pm(x,k)}$. The coefficient $b$ admits the representation
\[
b(k)=-\frac1{2ik}\int_\R q(t)m_-(t,k)e^{-2ikt}\dif t.
\]
Repeated integration by parts shows that $b$ is rapidly decreasing at infinity, together with all of its derivatives. The Volterra expansion in the formula
\[
a(k)=1+\frac1{2ik}\int_\R q(t)m_-(t,k)\dif t
\]
gives the stated decay estimates for $a-1$. In general $a-1$ is not a Schwartz function. Its leading term is $(2ik)^{-1}\int q$.

The zero-energy dichotomy and the limiting values in the generic case are classical \cite[Section 2]{DT79}. For the exceptional case write
\[
D_k(s)=\int_0^s e^{2ikr}\dif r,
\qquad
\partial_k^m D_k(s)=(2i)^m\int_0^s r^m e^{2ikr}\dif r,
\]
so that $|\partial_k^m D_k(s)|\leq 2^m|s|^{m+1}/(m+1)$. The $m$th $k$ derivative of the Volterra equation then involves at worst the moment $\int_\R(1+|t|)^{m+1}|q|\dif t$, which is finite for $q\in\SwR$. Successive differentiation and induction show that $m_\pm(x,k)$ is $C^\infty$ in the real variable $k$ near zero for each fixed $x$, hence so are the Wronskian numerators
\[
M(k)=\Wr\bigl(f_-(\cdot,k),f_+(\cdot,k)\bigr),
\qquad
N(k)=\Wr\bigl(f_+(\cdot,-k),f_-(\cdot,k)\bigr).
\]
In the exceptional case $M(0)=N(0)=0$, so $M(k)/k=\int_0^1M'(tk)\dif t$ and $N(k)/k=\int_0^1N'(tk)\dif t$ extend smoothly across the origin. Since $a=M/2ik$ and $b=N/2ik$, both $a$ and $b$ are $C^\infty$ at zero.
\end{proof}

\begin{lemma}\label{lem:zerodecay}
Let $q\in\SwRreal$ and let $y=f_+(\cdot,0)$. Then
\begin{equation}\label{eq:volt0}
y(x)=1-\int_x^{\infty}(t-x)\,q(t)\,y(t)\dif t .
\end{equation}
The function $y$ is bounded on $[0,\infty)$. For every $N\geq0$, there are constants $C_N$ and $C_{n,N}$ such that
\[
|y(x)-1|\leq C_N(1+x)^{-N},
\qquad
|y^{(n)}(x)|\leq C_{n,N}(1+x)^{-N}\quad(n\geq1)
\]
for all $x\geq0$. The analogous statements hold for $f_-(\cdot,0)$ on $(-\infty,0]$.
\end{lemma}

\begin{proof}
Equation \eqref{eq:volt0} is the zero-energy Volterra equation. Set
\[
\sigma(x)=\int_x^{\infty}(t-x)|q(t)|\dif t .
\]
For $x\geq0$ and every $N$,
\begin{equation}\label{eq:sigmadecay}
\sigma(x)\leq\int_x^{\infty}t|q(t)|\dif t\leq C_N(1+x)^{-N},
\end{equation}
so $\sigma$ decreases to zero faster than every power. In the Volterra series $y=\sum_{j\geq0}y_j$, with $y_0=1$ and
\[
y_{j+1}(x)=-\int_x^{\infty}(t-x)q(t)y_j(t)\dif t,
\]
induction gives $|y_j(x)|\leq\sigma(x)^j$. Choose $x_0\geq0$ so that $\sigma(x_0)\leq\tfrac12$. Then, for $x\geq x_0$,
\[
|y(x)-1|\leq\frac{\sigma(x)}{1-\sigma(x)}\leq2\sigma(x),
\qquad |y(x)|\leq2.
\]
The first estimate follows from \eqref{eq:sigmadecay}; after enlarging its constant it also holds on $[0,x_0]$. Thus $y$ is bounded on $[0,\infty)$.

Differentiating \eqref{eq:volt0} gives
\[
y'(x)=\int_x^{\infty}q(t)y(t)\dif t,
\]
which has rapid decay. For $n\geq2$, differentiate $y''=-qy$ a total of $n-2$ times. This gives
\[
y^{(n)}=-\sum_{j=0}^{n-2}\binom{n-2}{j}q^{(j)}y^{(n-2-j)}.
\]
Induction on $n$ proves rapid decay of every derivative. The proof for $f_-(\cdot,0)$ is identical after reversing the direction of integration.
\end{proof}

\begin{remark}\label{rem:moments}
The restriction to Schwartz decay is essential, since $W^{\infty,1}(\R)$ controls derivatives but not spatial moments. For $u=\sum_{n\geq2}n^{-2}\zeta(x-n)$ with $\zeta$ a fixed bump, $u\in W^{\infty,1}$ while the corresponding potential can have infinite first moment, so the zero-energy regularity used below need not hold in Lu's larger class.
\end{remark}

\begin{lemma}\label{lem:kexpansion}
Let $q\in\SwRreal$ be exceptional with $f_{-}(\cdot,0)=c\,f_{+}(\cdot,0)$. Then
\[
a_q(0)=\frac{c+c^{-1}}{2},\qquad b_q(0)=\frac{c-c^{-1}}{2},\qquad
T_q(0)=\frac{2c}{1+c^{2}},\qquad R_q(0)=\frac{c^{2}-1}{c^{2}+1}.
\]
In particular $a_q(0)\neq0$, one has $|R_q(0)|<1$, and $R_q(0)=0$ exactly when $c=\pm1$.
\end{lemma}

\begin{proof}
By Lemma \ref{lem:regularity}, the coefficients $a$ and $b$ are continuous at the origin, and $f_+(x,\pm k)\to f_+(x,0)$ as $k\to0$. Letting $k\to0$ in \eqref{eq:scatterrel} gives
\[
c\,f_+(x,0)=f_-(x,0)=\bigl(a_q(0)+b_q(0)\bigr)f_+(x,0),
\]
and $f_+(\cdot,0)$ does not vanish identically, so $a_q(0)+b_q(0)=c$. The symmetries $a(-k)=\overline{a(k)}$ and $b(-k)=\overline{b(k)}$ make $a_q(0)$ and $b_q(0)$ real, and letting $k\to0$ in $|a|^{2}=1+|b|^{2}$ gives
\[
a_q(0)^{2}-b_q(0)^{2}=1,
\qquad\text{hence}\qquad
a_q(0)-b_q(0)=c^{-1}.
\]
Solving these two linear equations gives the first two formulas, and $T_q(0)=1/a_q(0)$, $R_q(0)=b_q(0)/a_q(0)$ give the other two. Since $c$ is real and nonzero, $c+c^{-1}\neq0$, so $a_q(0)\neq0$; and $|c^{2}-1|<c^{2}+1$ gives $|R_q(0)|<1$, with $R_q(0)=0$ exactly when $c^{2}=1$.
\end{proof}

\begin{lemma}\label{lem:positive-c}
Let $q\in\SwRreal$ be exceptional and suppose that $H_q$ has no bound states. If $f_-(\cdot,0)=c\,f_+(\cdot,0)$, then $c>0$.
\end{lemma}

\begin{proof}
For real $q$, the analytic transmission denominator satisfies
\[
a(-\overline z)=\overline{a(z)},\qquad z\in\C_+,
\]
so $a(i\kappa)$ is real for $\kappa>0$. Its zeros in $\C_+$ are precisely the bound states. Hence $a(i\kappa)$ never vanishes, and since $a(i\kappa)\to1$ as $\kappa\to\infty$, one has $a(i\kappa)>0$ for every $\kappa>0$. Exceptionality gives continuity at the origin, so $a_q(0)\geq0$, and Lemma \ref{lem:kexpansion} gives $a_q(0)=(c+c^{-1})/2\neq0$. Hence $a_q(0)>0$ and therefore $c>0$.
\end{proof}

\begin{theorem}[Threshold scattering formula]\label{thm:threshold}
Let $w\in\SwRreal$, set
\[
q=w'-w^2,
\qquad
I(w)=\int_\R w(x)\dif x.
\]
Then $H_q\geq0$, the potential is exceptional, and
\[
f_+(x,0)=\exp\!\left(\int_x^\infty w(t)\dif t\right),
\qquad
f_-(x,0)=\exp\!\left(-\int_{-\infty}^x w(t)\dif t\right).
\]
Consequently
\begin{equation}\label{eq:threshold-mass}
\begin{aligned}
a_q(0)&=\cosh I(w), & b_q(0)&=-\sinh I(w),\\
T_q(0)&=\operatorname{sech} I(w), & R_q(0)&=-\tanh I(w).
\end{aligned}
\end{equation}
\end{theorem}

\begin{proof}
The factorization
\[
H_q=(-\partial_x+w)(\partial_x+w)
\]
gives nonnegativity. Both displayed functions solve $(\partial_x+w)y=0$ and are bounded, the first tending to $1$ as $x\to+\infty$ and the second as $x\to-\infty$. A linearly independent zero-energy solution has the form $y(x)\int^x y(t)^{-2}\dif t$ and grows linearly there. Hence the displayed solutions are the zero-energy Jost solutions. They satisfy
\[
f_-(\cdot,0)=e^{-I(w)}f_+(\cdot,0),
\]
so the potential is exceptional with $c=e^{-I(w)}$. Substitution into Lemma \ref{lem:kexpansion} gives \eqref{eq:threshold-mass}.
\end{proof}

\section{The image of the real Bers map}\label{sec:image}

\begin{theorem}\label{thm:image}
Let $q\in\SwRreal$. The following are equivalent.
\begin{enumerate}
\item $q=\Bers(\varphi)$ for a necessarily unique $\varphi\in\Dg$.
\item $q=w'-w^2$ for a necessarily unique $w\in\SwRzero$.
\item $H_q$ has no bound states and $R_q(0)=0$.
\item $H_q$ has no bound states and $T_q(0)=1$.
\end{enumerate}
When these conditions hold, $f_+(\cdot,0)$ is positive, $f_+(x,0)\to1$ as $x\to\pm\infty$, and
\[
\varphi'(x)=f_+(x,0)^{-2}.
\]
\end{theorem}

\begin{proof}
The equivalence of (1) and (2) is Lemma \ref{lem:riccati-chart}. If (2) holds, Theorem \ref{thm:threshold} gives $R_q(0)=0$ and $T_q(0)=1$, while the factorization of $H_q$ excludes bound states. Thus (2) implies (3) and (4).

Assume (3). The generic zero-energy case has $R_q(0)=-1$ by Lemma \ref{lem:regularity}, so $q$ is exceptional. Write
\[
f_-(\cdot,0)=c\,f_+(\cdot,0).
\]
Lemma \ref{lem:kexpansion} gives $c=\pm1$, and Lemma \ref{lem:positive-c} gives $c>0$. Hence $c=1$. The same conclusion follows from (4), since Lemma \ref{lem:kexpansion} and Lemma \ref{lem:positive-c} show that $T_q(0)=1$ is equivalent to $c=1$.

It remains to reconstruct the diffeomorphism from $c=1$. Set $y=f_+(\cdot,0)=f_-(\cdot,0)$. Then $y(x)\to1$ as $x\to\pm\infty$. We claim that $y$ has no zero. Every zero is simple, because a solution of a second-order equation that vanishes together with its derivative is identically zero. Since the limits at $+\infty$ and at $-\infty$ are positive, any zero forces at least two zeros. Choose consecutive zeros $a<b$ and define $h=y$ on $[a,b]$ and $h=0$ outside. Then $h\in H^1(\R)$ and
\[
\mathcal Q_q[h]=\int_\R\bigl(|h'|^2-qh^2\bigr)\dif x=0.
\]
The form $\mathcal Q_q$ is nonnegative. Therefore, for every $g\in C_c^\infty(\R)$, the linear term in $\mathcal Q_q[h+t g]$ vanishes. Integration by parts gives
\[
0=\int_\R(h'g'-qh g)\dif x
  =y'(b)g(b)-y'(a)g(a).
\]
The values $g(a)$ and $g(b)$ may be chosen independently, so $y'(a)=y'(b)=0$, contradicting simplicity. Hence $y>0$.

By Lemma \ref{lem:zerodecay}, $y-1$ and all derivatives of $y$ decay rapidly as $x\to\pm\infty$. Since $y$ is positive and tends to one, it is bounded above and away from zero. Thus
\[
u=-2\log y\in\SwRreal.
\]
Define
\[
\varphi(x)=x+\int_{-\infty}^x\bigl(y(t)^{-2}-1\bigr)\dif t.
\]
Then $\varphi\in\Dg$, $\varphi'=e^u=y^{-2}$, and the equation $y''+qy=0$ gives
\[
\Bers(\varphi)=\frac12u''-\frac14(u')^2=-\frac{y''}{y}=q.
\]
This proves (1). Uniqueness follows from Proposition \ref{prop:injective}.
\end{proof}

\section{Reflection coordinates}\label{sec:chart}

Recall $\RRset=\bigl\{R\in\SwR\ \big|\ \overline{R(-k)}=R(k),\ \lVert R\rVert_\infty<1,\ R(0)=0\bigr\}$.

\begin{corollary}\label{cor:reflection}
The scattering map
\[
\sla\colon\Dg\longrightarrow\RRset,
\qquad
\sla(\varphi)=R_{\Bers(\varphi)},
\]
is a bijection.
\end{corollary}

\begin{proof}
Let $\varphi\in\Dg$, put $w=\tfrac12(\log\varphi')'$, and set $q=\Bers(\varphi)=w'-w^2$. The direct scattering estimates of Section \ref{sec:jost} give $R_q\in\SwR$ and the reality symmetry. For $k\neq0$, the identity $|T_q(k)|^2+|R_q(k)|^2=1$ and $T_q(k)\neq0$ give $|R_q(k)|<1$, while Theorem \ref{thm:threshold} gives $R_q(0)=0$. Since $R_q$ is continuous and tends to zero at infinity, $\lVert R_q\rVert_\infty<1$. If two elements of $\Dg$ have the same reflection coefficient, then by Theorem \ref{thm:image} both potentials are exceptional without bound states, so the scattering data consist of the reflection coefficient alone and the Marchenko theory on the line \cite[Chapter 3]{Mar86}, \cite{DT79} gives the same potential; Proposition \ref{prop:injective} then gives the same diffeomorphism.

For surjectivity, let $R\in\RRset$ and set $\widetilde R(s)=R(s/2)$. In the ZS-AKNS normalization of \cite{FHMP09}, the spectral variables are related by $s=2k$. Their right reflection coefficient is
\[
r_+(s)=-\frac{\overline{b_{\rm AKNS}(s)}}{a_{\rm AKNS}(s)}.
\]
The asymptotics in \cite[(1.9) and (1.10)]{FHMP09} show that $\widetilde R=r_+$ for the Schr\"odinger operator $H_q=-\partial_x^2-q$. Since $\widetilde R$ is Schwartz and satisfies $\overline{\widetilde R(-s)}=\widetilde R(s)$, its inverse Fourier transform is real and belongs to $L^1(\R)\cap L^2(\R)$; moreover $\lVert\widetilde R\rVert_\infty<1$. Theorem 1.1 and Corollary 1.2 of \cite{FHMP09} therefore give a unique real Riccati representative $v\in L^1(\R)\cap L^2(\R)$ for a no-bound-state potential
\[
V=v'+v^2
\]
in the standard convention $-\partial_x^2+V$. In this reduction $v$ is the potential of the associated ZS-AKNS system and $\widetilde R$ is its reflection coefficient, so the weighted Sobolev theory applies to $v$ directly. Let $H^{j,k}(\R)$ consist of the $f\in L^2(\R)$ with $f^{(j)}\in L^2(\R)$ and $(1+|x|)^{k}f\in L^2(\R)$. For $j\geq0$ and $k\geq1$ the ZS-AKNS scattering map is a bijection from real $H^{j,k}(\R)$ onto the reflection coefficients of class $H^{k,j}(\R)$ of supremum norm less than one \cite{Zhou98}. Since $\widetilde R$ is Schwartz it lies in $H^{k,j}(\R)$ for every such pair, so each pair produces a real $v_{j,k}\in H^{j,k}(\R)$ with reflection coefficient $\widetilde R$. The weight $k\geq1$ and Cauchy\textendash Schwarz give $H^{j,k}(\R)\subset L^1(\R)\cap L^2(\R)$, so the injectivity in \cite[Theorem 1.1]{FHMP09} forces $v_{j,k}=v$ for every $j$ and $k$. Since $\bigcap_{j,k}H^{j,k}(\R)=\SwR$, we conclude $v\in\SwRreal$. The exceptional mapping property on the Schr\"odinger side, obtained from the same techniques, is \cite[Theorem 1.1]{HMP10}.

Put $w=-v$ and $q=-V=w'-w^2$. Then $H_q=-\partial_x^2+V$ has right reflection coefficient $R$. Theorem \ref{thm:threshold} and $R_q(0)=R(0)=0$ give
\[
\int_\R w\dif x=0.
\]
Thus $w\in\SwRzero$, and Lemma \ref{lem:riccati-chart} produces a unique $\varphi\in\Dg$ with $\Bers(\varphi)=q$. Hence $\sla(\varphi)=R$.
\end{proof}

\section{The differential of the real Bers map}\label{sec:cokernel}

Throughout this section fix $u\in\SwRreal$, let $q=\tfrac12u''-\tfrac14(u')^{2}$, $\psi=e^{-u/2}$, and let
\[
\widetilde\psi(x)=\psi(x)\int_{x_0}^{x}\frac{\dif t}{\psi(t)^{2}},\qquad \Wr(\psi,\widetilde\psi)=1 ,
\]
be the second zero energy solution, which grows linearly as $x\to\pm\infty$, since $\psi^{-2}=e^{u}\to1$. In the chart of Lemma \ref{lem:chart} the Bers map is $u\mapsto\tfrac12u''-\tfrac14(u')^2$ of $\SwRreal$ to itself, with differential
\begin{equation}\label{eq:dbers}
D\Bers_u(\delta u)=\tfrac12\big(\delta u''-u'\,\delta u'\big)
=\tfrac12\,e^{u}\,\partial_x\big(e^{-u}\,\partial_x\,\delta u\big)
=\tfrac12(\partial_x-u')\partial_x\,\delta u .
\end{equation}

\begin{theorem}\label{thm:differential}
\begin{enumerate}
\item For $v\in\SwRreal$ the equation $D\Bers_u(\delta u)=v$ has a solution $\delta u\in\SwRreal$ if and only if
\begin{equation}\label{eq:twoconditions}
\Lambda_1(v)=\int_\R \psi^{2}v\dif x=0
\qquad\text{and}\qquad
\Lambda_2(v)=\int_\R \psi\widetilde\psi\,v\dif x=0 ,
\end{equation}
and the solution is then unique. The functional $\Lambda_2$ does not depend on the base point $x_0$ on the kernel of $\Lambda_1$. Both functionals are continuous and linearly independent on $\Sw$. Hence the image of $D\Bers_u$ is closed and split of codimension two, with annihilator spanned by $\psi^{2}$ and $\psi\widetilde\psi$. On this image the inverse is
\begin{equation}\label{eq:inverse_dbers}
G_uv(x)=2\int_{-\infty}^{x}e^{u(t)}\int_{-\infty}^{t}e^{-u(y)}v(y)\dif y\dif t.
\end{equation}
\item For $v\in\SwRreal$ and the straight line variation $\widehat q_\varepsilon=q+\varepsilon v$,
\[
\frac{\dif}{\dif\varepsilon}\Big|_{0}\Wr\big(\widehat f^{\varepsilon}_-(\cdot,0),\widehat f^{\varepsilon}_+(\cdot,0)\big)=\Lambda_1(v),
\]
so $\Lambda_1$ measures the first order destruction of exceptionality at zero energy. Moreover, for every $v$ with $\Lambda_1(v)=0$ there is an explicit smooth curve $\varepsilon\mapsto q_\varepsilon\in\SwRreal$ of exceptional potentials without bound states, with $q_0=q$ and $\dot q_0=v$, along which
\[
c(q_\varepsilon)=1+\varepsilon\,\Lambda_2(v)+O(\varepsilon^{2}),
\qquad\text{hence}\qquad
R_{q_\varepsilon}(0)=\varepsilon\,\Lambda_2(v)+O(\varepsilon^{2}) ,
\]
so $\Lambda_2$ measures the first order motion of $R_q(0)$ away from zero.
\item For every integer $k\geq0$, let
\[
L_{u,k}=D\Bers_u\colon W^{k+2,1}(\R)\longrightarrow W^{k,1}(\R).
\]
Then
\[
\overline{\operatorname{Ran}L_{u,k}}=\ker\Lambda_1,
\]
but $\operatorname{Ran}L_{u,k}$ is not closed. Hence the Sobolev realization is not split, and the inverse of $L_{u,k}$ on its range is unbounded.
\end{enumerate}
\end{theorem}

\begin{proof}
We prove (1). Write $v=D\Bers_u(\delta u)$. Since $e^{-u}(\delta u''-u'\delta u')=(e^{-u}\delta u')'$ and $\psi^{2}=e^{-u}$,
\[
\big(\psi^{2}\,\delta u'\big)'=2\psi^{2}v .
\]
If $\delta u\in\Sw$, integrating over the line gives $\Lambda_1(v)=0$, and then
\begin{equation}\label{eq:solveg}
\delta u'(x)=2\,\psi(x)^{-2}\int_{-\infty}^{x}\psi^{2}v\dif t=-2\,\psi(x)^{-2}\int_{x}^{\infty}\psi^{2}v\dif t ,
\end{equation}
which lies in $\Sw$, the first expression decaying rapidly as $x\to-\infty$ and the second as $x\to+\infty$. Finally $\delta u=\int_{-\infty}^{x}\delta u'$ lies in $\Sw$ exactly when $\int_\R\delta u'=0$. To compute this integral, split at the base point $x_0$ of $\widetilde\psi$, use the first expression in \eqref{eq:solveg} on $(-\infty,x_0)$ and the second on $(x_0,\infty)$, and apply Fubini on the regions $t<x<x_0$ and $x_0<x<t$; the double integrals converge absolutely because $\int(1+|t|)\psi^{2}|v|<\infty$ for $v\in\Sw$ while $\int_{x_0}^{t}\psi^{-2}=O(|t|)$. This yields
\[
\int_\R\delta u'\dif x
=2\int_\R \psi(t)^{2}v(t)\Big(-\int_{x_0}^{t}\psi(x)^{-2}\dif x\Big)\dif t
=-2\int_\R v\,\psi\,\widetilde\psi\dif t=-2\Lambda_2(v).
\]
On the kernel of $\Lambda_1$, a change of base point shifts $\widetilde\psi$ by a multiple of $\psi$ and hence $\Lambda_2$ by a multiple of $\Lambda_1(v)=0$. Thus $\delta u\in\Sw$ exactly when \eqref{eq:twoconditions} holds, uniqueness being built into \eqref{eq:solveg} together with the normalization $\delta u(-\infty)=0$. Continuity follows because $\psi^{2}$ is bounded and $\psi\widetilde\psi=O(|x|)$. The function $\widetilde\psi/\psi$ is strictly increasing, so the two functionals are linearly independent. Choose $\chi_1,\chi_2\in\SwRreal$ so that the matrix $\bigl(\Lambda_i(\chi_j)\bigr)_{i,j=1}^{2}$ is invertible. Subtracting the corresponding linear combination of $\chi_1$ and $\chi_2$ gives a continuous projection onto the common kernel, which proves that the range is split. Since $\psi^{2}=e^{-u}$, the first expression in \eqref{eq:solveg} reads $\delta u'(t)=2e^{u(t)}\int_{-\infty}^{t}e^{-u(y)}v(y)\dif y$, and integrating it from $-\infty$ with $\delta u(-\infty)=0$ gives \eqref{eq:inverse_dbers}. The same estimates show that $G_u$ is continuous from the indicated closed subspace of $\Sw$ to $\Sw$.

We prove (2). Let
\[
\delta f_\pm=\frac{\dif}{\dif\varepsilon}\Big|_0
\widehat f^{\varepsilon}_\pm(\cdot,0).
\]
Differentiation of the Volterra equations in $\varepsilon$ gives
\[
\delta f_\pm''+q\,\delta f_\pm=-v\,f_\pm(\cdot,0),
\]
with $\delta f_+$ and $\delta f_+'$ tending to zero at $+\infty$, and $\delta f_-$ and $\delta f_-'$ at $-\infty$. Put $y_\pm=f_\pm(\cdot,0)$, so $y_-=y_+=\psi$. Since
\[
\bigl(\Wr(g,h)\bigr)'=g h''-g''h,
\]
we obtain
\[
\Wr(y_-,\delta f_+)(x)=\int_x^\infty v\,y_-y_+\dif t,
\qquad
\Wr(\delta f_-,y_+)(x)=\int_{-\infty}^x v\,y_-y_+\dif t.
\]
Adding,
\[
\frac{\dif}{\dif\varepsilon}\Big|_0\Wr(\widehat f_-^{\varepsilon},\widehat f_+^{\varepsilon})
=\int_{-\infty}^{x}v\,\psi^{2}\dif t+\int_{x}^{\infty}v\,\psi^{2}\dif t
=\Lambda_1(v).
\]

For the curve construction let $v\in\SwRreal$ with $\Lambda_1(v)=0$ and define, with the Green kernel $G(x,t)=\psi(x)\widetilde\psi(t)-\widetilde\psi(x)\psi(t)$ of the zero energy operator, for which $y(x)=\int_x^\infty G(x,t)g(t)\dif t$ solves $y''+qy=g$ and vanishes with its derivative at $+\infty$,
\[
h(x)=-\psi(x)\int_{x}^{\infty}\widetilde\psi\,\psi\,v\dif t+\widetilde\psi(x)\int_{x}^{\infty}\psi^{2}\,v\dif t ,
\]
the solution of $h''+q\,h=-v\,\psi$ with $h,h'\to0$ at $+\infty$; the integrals converge absolutely. As $x\to-\infty$ the first term tends to $-\Lambda_2(v)$, since $\psi(x)\to1$ as $x\to\pm\infty$, while in the second term $\int_x^\infty\psi^2 v=\Lambda_1(v)-\int_{-\infty}^x\psi^{2}v=O(|x|^{-\infty})$ kills the linear growth of $\widetilde\psi$; moreover $h$ approaches its limits rapidly with all derivatives, as one sees by iterating $h''=-qh-v\psi$. Now set
\[
y_\varepsilon=\psi+\varepsilon h,\qquad
q_\varepsilon=-\frac{y_\varepsilon''}{y_\varepsilon}=q+\varepsilon\,v\,\frac{\psi}{y_\varepsilon},
\]
the second equality because $y_\varepsilon''=\psi''+\varepsilon h''=-q\psi+\varepsilon(-qh-v\psi)=-q\,y_\varepsilon-\varepsilon v\psi$. For small $\varepsilon$, we have $y_\varepsilon>0$, since $\psi\geq m>0$ and $h$ is bounded, and $q_\varepsilon\in\SwRreal$ with $q_0=q$, $\dot q_0=v$, the curve being smooth in $\varepsilon$. By construction, $y_\varepsilon$ is a positive zero-energy solution of $H_{q_\varepsilon}y_\varepsilon=0$, and the substitution $f=y_\varepsilon g$ gives the identity
\[
\int_\R\big((f')^{2}-q_\varepsilon f^{2}\big)\dif x=\int_\R y_\varepsilon^{2}\,\big|(f/y_\varepsilon)'\big|^{2}\dif x\;\geq\;0
\qquad (f\in C_c^{\infty}),
\]
so $H_{q_\varepsilon}\geq0$ and there are no bound states. Since $y_\varepsilon\to1$ at $+\infty$ and $y_\varepsilon\to1-\varepsilon\Lambda_2(v)$ at $-\infty$, the Wronskian argument in the proof of Theorem \ref{thm:image} identifies $f_+^{\varepsilon}(\cdot,0)=y_\varepsilon$, so $q_\varepsilon$ is exceptional with
\[
\frac1{c(q_\varepsilon)}=\lim_{x\to-\infty}y_\varepsilon(x)=1-\varepsilon\Lambda_2(v),
\]
which gives $c(q_\varepsilon)=1+\varepsilon\Lambda_2(v)+O(\varepsilon^{2})$, and Lemma \ref{lem:kexpansion} gives $R_{q_\varepsilon}(0)=\varepsilon\Lambda_2(v)+O(\varepsilon^{2})$.

We prove (3). The identity
\[
\bigl(e^{-u}\delta u'\bigr)'=2e^{-u}L_{u,k}\delta u
\]
shows that $\operatorname{Ran}L_{u,k}\subseteq\ker\Lambda_1$. Choose $\eta,\sigma\in C_c^\infty(\R)$ with $\int\eta\neq0$ and $\Lambda_1(\sigma)=1$, and set
\[
h_n=\eta(\cdot-n)-\Lambda_1\bigl(\eta(\cdot-n)\bigr)\sigma.
\]
Then $h_n\in\Sw\cap\ker\Lambda_1$ and the norms $\lVert h_n\rVert_{W^{k,1}}$ are bounded. Since
\[
\psi(x)\widetilde\psi(x)=x+O(1),
\qquad
\psi(x)^2=1+o(1)
\quad (x\to+\infty),
\]
one has
\[
\Lambda_2(h_n)=n\int_\R\eta\dif x+O(1).
\]
Thus $\Lambda_2$ is not continuous in the $W^{k,1}$-norm.

Let $v\in W^{k,1}(\R)$ satisfy $\Lambda_1(v)=0$. Choose $g_j\in C_c^\infty(\R)$ with $g_j\to v$ in $W^{k,1}$ and put
\[
\widetilde g_j=g_j-\Lambda_1(g_j)\sigma.
\]
Then $\widetilde g_j\to v$ and $\Lambda_1(\widetilde g_j)=0$. Choose $n_j$ so large that
\[
\left|\frac{\Lambda_2(\widetilde g_j)}{\Lambda_2(h_{n_j})}\right|
\lVert h_{n_j}\rVert_{W^{k,1}}<\frac1j,
\]
and define
\[
v_j=\widetilde g_j-
\frac{\Lambda_2(\widetilde g_j)}{\Lambda_2(h_{n_j})}h_{n_j}.
\]
Then $v_j\in\Sw$, $\Lambda_1(v_j)=\Lambda_2(v_j)=0$, and $v_j\to v$ in $W^{k,1}$. Part (1) gives $v_j\in\operatorname{Ran}L_{u,k}$. Therefore
\[
\ker\Lambda_1\subseteq\overline{\operatorname{Ran}L_{u,k}}.
\]
The reverse inclusion was already proved.

Finally choose $v_0\in\Sw\cap\ker\Lambda_1$ with $\Lambda_2(v_0)\neq0$. The preceding approximation puts $v_0$ in the closure of the range. If $L_{u,k}\delta u=v_0$ for some $\delta u\in W^{k+2,1}$, then $\delta u$ and $\delta u'$ tend to zero as $x\to\pm\infty$. Integrating as in part (1) gives
\[
0=\int_\R\delta u'\dif x=-2\Lambda_2(v_0),
\]
a contradiction. The range is therefore not closed. A continuously complemented subspace of a Banach space is closed, so the range is not split. If the inverse on the range were bounded, the range would be complete in the inherited norm and hence closed. This proves the last assertions.
\end{proof}

\begin{remark}
The two conditions in \eqref{eq:twoconditions} arise from the two integrations in \eqref{eq:inverse_dbers}. If $\Lambda_1(v)\neq0$, then $e^{-u}\delta u'$ has a nonzero limit and $\delta u'\notin L^1$. If $\Lambda_1(v)=0$ but $\Lambda_2(v)\neq0$, then $\delta u'$ is Schwartz while $\delta u$ has a nonzero limit at $+\infty$. Both obstructions are intrinsic.
\end{remark}

\end{document}